\documentclass[11pt]{amsart}
\usepackage{amscd}
\usepackage[all]{xy}
\usepackage{graphicx}
\usepackage{amsmath}
\usepackage{amsfonts}
\usepackage{amssymb}
\usepackage{latexsym}
\usepackage{slashed}
\usepackage{soul}
\usepackage{comment}
\usepackage{color}
\usepackage{extarrows}
\usepackage{mathtools}

\numberwithin{equation}{section}
\theoremstyle{plain}
\newtheorem{lemma}{Lemma}[section]
\newtheorem{proposition}[lemma]{Proposition}
\newtheorem{theorem}[lemma]{Theorem}

\theoremstyle{definition}
\newtheorem{definition}[lemma]{Definition}
\newtheorem{remark}[lemma]{Remark}
\newtheorem{example}[lemma]{Example}

\newcommand{\R}{{\mathbb R}}
\newcommand{\Z}{{\mathbb Z}} 
\newcommand{\N}{{\mathbb N}}
\newcommand{\Hh}{{\mathcal H}}
\newcommand{\Ll}{{\mathcal L}}

\newcommand{\Om}{{\Omega}}
\newcommand{\om}{{\omega}}
\newcommand{\X}{{\mathfrak X}}

\newcommand{\p}{{\partial}}

\newcommand{\Sp}{{\text{\rm Spin}(7)}}
\renewcommand{\Im}{{ \rm Im \,}}
\renewcommand{\frak}[1]{{\mathfrak {#1}}}

\newcommand{\eps}{{\varepsilon}}

\renewcommand{\a}{{\mathfrak a}}
\newcommand{\m}{{\mathfrak m}}

\begin{document}

\date{\today}
\title[Fr\"olicher-Nijenhuis bracket]
{Fr\"olicher-Nijenhuis bracket on manifolds with special holonomy}

\author{Kotaro Kawai}
\address{Gakushuin University, 1-5-1, Mejiro, Toshima,Tokyo, 171-8588, Japan}
\email{kkawai@math.gakushuin.ac.jp}

\author{H\^ong V\^an L\^e}
\address{Institute of Mathematics, Czech Academy of Sciences, Zitna 25, 11567 Praha 1, Czech Republic}
\email{hvle@math.cas.cz}

\author{Lorenz Schwachh\"ofer}
\address{Fakult\"at f\"ur Mathematik,
TU Dortmund University,
Vogelpothsweg 87, 44221 Dortmund, Germany}
\email{lschwach@math.tu-dortmund.de} 

\thanks{The first named author is supported by 
JSPS KAKENHI Grant Numbers JP17K14181, and 
the research of the second named author was  supported 
by the  GA\v CR  project  18-00496S  and  RVO:67985840. The third named author was supported by the Deutsche Forschungsgemeinschaft by grant SCHW 893/5-1.}

\begin{abstract} 
In this article, we summarize our recent results 
on the study of manifolds with special holonomy via the Fr\"olicher-Nijenhuis bracket. 
This bracket enables 
us to define 
the Fr\"olicher-Nijenhuis cohomologies which are analogues of 
the $d^c$ and the Dolbeault cohomologies in K\"ahler geometry, 
and 
assigns an $L_\infty$-algebra 
to each associative  submanifold. 
We provide   several concrete   computations of the Fr\"olicher-Nijenhuis  cohomology. 

\end{abstract}

\maketitle

%\tableofcontents
%%%%%%%%%%%%%%%%%%%%%%%%%%%%%%%%%%%%%%%%%%%%%%%%%%%%%%%%%
\section{Introduction}
%%%%%%%%%%%%%%%%%%%%%%%%%%%%%%%%%%%%%%%%%%%%%%%%%%%%%%%%%

The Fr\"olicher-Nijenhuis bracket, 
which was introduced in \cite{FN1956, FN1956b},  
defines a natural structure of a graded Lie algebra 
on the space of tangent bundle valued differential forms
$\Om^\ast(M, TM)$ on a smooth manifold $M$. 

On a Riemannian manifold $(M,g)$, 
if there is a parallel differential form of even degree, 
we can define canonical cohomologies 
which are analogues of 
the $d^c$ and the Dolbeault cohomologies in K\"ahler geometry. 
See Section \ref{subs:parallel}. 
We compute these cohomologies for $G_2$-manifolds in Theorem \ref{thm:hG2} 
and give a sketch of the proof in Section \ref{sec:g2}. 
A similar statement holds for $\Sp$ and Calabi-Yau manifolds. 
See \cite{KLS2,KLS3}. 

In the second part  of our  note, using the  Fr\"olicher-Nijenhuis bracket, we  assign  to each associative 
submanifold  an $L_\infty$-algebra.

{\bf Notation}: 
Let $(V, g)$ be an $n$-dimensional oriented real vector space with a scalar product $g$. 
Define the map $\p$ by contraction of a form with the metric $g$, i.e.
\begin{equation} \label{eq:def-partial}
\p = \p_g: \Lambda^k V^\ast \longrightarrow \Lambda^{k-1} V^\ast \otimes V, \qquad \p_g(\alpha^k) := (\imath_{e_i} \alpha^k) \otimes e^i,
\end{equation}
where $(e_i)$ is an  orthonormal basis of $V$ with the dual basis $(e^i)$.

%%%%%%%%%%%%%%%%%%%%%%%%%%%%%%%%%%%%%%%%%%%%%%%%%%%%%%%%%
\section{Preliminaries}\label{sec:pre}
%%%%%%%%%%%%%%%%%%%%%%%%%%%%%%%%%%%%%%%%%%%%%%%%%%%%%%%%%

%%%%%%%%%%%%%%%%%%%%%%%%%%%%%%%%%%%%%%%%%%%%%%%%%%%%%%%%%
\subsection{Graded Lie algebras and differentials}\label{sec:coh}
%%%%%%%%%%%%%%%%%%%%%%%%%%%%%%%%%%%%%%%%%%%%%%%%%%%%%%%%%

We briefly recall some basic notions and properties of graded (Lie) algebras. 
Let $V := (\bigoplus_{k \in \Z} V_k, \cdot)$ be a graded real vector space with a graded bilinear map $\cdot: V \times V \to V$, called a {\em product on $V$}. A {\em graded derivation of $(V, \cdot)$ of degree $l$ }
is a linear map $D^l: V \to V$ of degree $l$ (i.e., $D^l (V_k) \subset V_{k+l}$) such that 
\begin{equation} \label{eq:derivations}
D^l(x \cdot y) = (D^l x) \cdot y + (-1)^{l|x|} x \cdot (D^l y),
\end{equation}
where $|x|$ denotes the degree of an element, i.e. $|x| = k$ for $x \in V_k$.
If we denote by ${\mathcal D}^l(V)$ the graded derivations of $(V, \cdot)$ of degree $l$, then ${\mathcal D}(V) := \bigoplus_{l \in \Z} {\mathcal D}^l(V)$ is a graded Lie algebra with the Lie bracket
\begin{equation} \label{eq:derivations-bracket}
{}[D_1, D_2] := D_1 D_2 - (-1)^{|D_1| |D_2|} D_2 D_1,
\end{equation}
i.e., the Lie bracket is graded anti-symmetric and satisfies the graded Jacobi identity,
\begin{eqnarray}
\label{GLA-skew} && [x, y] = -(-1)^{|x||y|} [y,x]\\
\label{GLA-Jac} && (-1)^{|x||z|} [x, [y,z]] + (-1)^{|y||x|} [y, [z,x]] + (-1)^{|z||y|} [z, [x,y]] = 0.
\end{eqnarray}

In general, if $L = (\bigoplus_{k \in \Z} L_k, [\cdot, \cdot])$ is a graded Lie algebra, then an {\em action of $L$ on $V$ }is a Lie algebra homomorphism $\pi: L \to {\mathcal D}(V)$, which yields a graded bilinear map $L \times V \to V$, $(x, v) \mapsto \pi(x)(v)$ such that the map $\pi(x): V \to V$ is a graded derivation of degree $|x|$ and such that
\[
{}[\pi(x), \pi(y)] = \pi[x,y].
\]
For instance, a graded Lie algebra acts on itself via the adjoint representation $ad: L \to {\mathcal D}(L)$, where $ad_x(y) := [x,y]$.

For a graded Lie algebra $L$ we define the set of {\em Maurer-Cartan elements of $L$ of degree $2k+1$ }as
\[
{\mathcal {MC}}^{2k+1}(L) := \{ \xi \in L_{2k+1} \mid [\xi,\xi] = 0\}.
\]
If $\pi: L \to {\mathcal D}(V)$ is an action of $L$ on $(V, \cdot)$, then for $\xi \in {\mathcal {MC}}^{2k+1}(L)$ we have $0 = [\pi(\xi), \pi(\xi)] = 2 \pi(\xi)^2$, so that $\pi(\xi): V \to V$ is a differential on $V$. We define the {\em cohomology of $(V, \cdot)$ w.r.t. $\xi$ }as
\begin{equation} \label{eq:cohom-GA}
H^i_\xi(V): = 
\frac{\ker (\pi(\xi): V_i \to V_{i+2k+1})}{\Im (\pi(\xi): V_{i-(2k+1)} \to V_i)} \qquad \mbox{for $\xi \in {\mathcal {MC}}^{2k+1}(L)$.}
\end{equation}
Since $\pi(\xi)$ is a derivation, it follows that $\ker \pi(\xi) \cdot \ker \pi(\xi) \subset \ker \pi(\xi)$, whence there is an induced product on $H^\ast_\xi(V) := \bigoplus_{i \in \Z} H^i_\xi(V)$.

If $L = \bigoplus_{k \in \Z} L_k$ is a graded Lie algebra, then for $v \in L_0$ and $t \in \R$, we define the formal power series
\begin{equation} \label{eq:formal-exp}
\exp(tv): L \longrightarrow L[[t]], \qquad \exp(tv)(x) := \sum_{k=0}^\infty \dfrac{t^k}{k!} ad_v^k(x).
\end{equation}
Observe that $ad_\xi(v) = 0$ for some $v \in L_0$ iff $ad_v(\xi) = 0$ iff $\exp(tv)(\xi) = \xi$ for all $t \in \R$. In this case, we call $v$ an {\em infinitesimal stabilizer of $\xi$}.

For $\xi \in {\mathcal {MC}}^{2k+1}(L)$, we say that $x \in L_{2k+1}$ is an {\em infinitesimal deformation of $\xi$ within ${\mathcal {MC}}^{2k+1}(L)$ }if $[\xi + t x, \xi + t x] = 0 \mod t^2$. Evidently, this is equivalent to $[\xi, x] = 0$ or $x \in \ker ad_\xi$. Such an infinitesimal deformation is called {\em trivial} if $x = [\xi, v]$ for some $v \in L_0$, since in this case, $\xi + t x = \exp(-tv)(\xi) \mod t^2$, whence up to second order, it coincides with elements in the orbit of $\xi$ under the (formal) action of $\exp(tv)$.
Thus, we have the following interpretation of some cohomology groups.

\begin{proposition} \label{prop:homology-general}
Let $(L = \bigoplus_{i \in \Z} L_i, [\cdot, \cdot])$ be a real graded Lie algebra, acting on itself by the adjoint representation, and let $\xi \in {\mathcal {MC}}^{2k+1}(L)$. Then the following holds.
\begin{enumerate}
\item If $L_{-(2k+1)} = 0$, then $H_\xi^0(L)$ is the Lie algebra of infinitesimal stabilizers of $\xi$.
\item $H_\xi^{2k+1}(L)$ is the space of infinitesimal deformations of $\xi$ within \linebreak ${\mathcal {MC}}^{2k+1}(L)$ modulo trivial deformations.
\end{enumerate}
\end{proposition}

%%%%%%%%%%%%%%%%%%%%%%%%%%%%%%%%%%%%%%%%%%%%%%%%%%%%%%%%%
\subsection{The Fr\"olicher-Nijenhuis bracket}\label{subs:fnb}
%%%%%%%%%%%%%%%%%%%%%%%%%%%%%%%%%%%%%%%%%%%%%%%%%%%%%%%%%

We shall apply our discussion from the preceding section to the following example. Let $M$ be a manifold and $(\Om^\ast(M), \wedge) = (\bigoplus_{k \geq 0} \Om^k(M), \wedge)$ be the graded algebra of differential forms. Evidently, the exterior derivative $d: \Om^k(M) \to \Om^{k+1}(M)$ is a derivation of $\Om^\ast(M)$ of degree $1$, whereas insertion $\imath_X: \Om^k(M) \to \Om^{k-1}(M)$ of a vector field $X \in {\frak X}(M)$ is a derivation of degree $-1$.

More generally, for $K \in \Om^k(M, TM)$ we define $\imath_K \alpha^l$ as the {\em insertion of $K$ into $\alpha^l \in \Om^l(M)$ }pointwise by
\[
\imath_{\kappa^k \otimes X} \alpha^l := \kappa^k \wedge (\imath_X \alpha^l) \in \Om^{k+l-1}(M),
\]
where $\kappa^k \in \Om^k(M)$ and $X \in {\frak X}(M)$, and this is a derivation of $\Om^\ast(M)$ of degree $k-1$. Thus, the {\em Nijenhuis-Lie derivative along $K \in \Om^k(M, TM)$ }defined as
\begin{equation} \label{eq:LK-deriv}
\Ll_K (\alpha^l) := [\imath_K, d] (\alpha^l) = \imath_K (d\alpha^l) + (-1)^k d(\imath_K \alpha^l) \in \Om^{k+l}(M)
\end{equation}
is a derivation of $\Om^\ast(M)$ of degree $k$. 

Observe that for $k = 0$ in which case $K \in \Om^0(M, TM)$ is a vector field, both $\imath_K$ and $\Ll_K$ coincide with the standard notion of insertion of and Lie derivative along a vector field.

In \cite{FN1956, FN1956b}, it was shown that $\Om^\ast(M, TM)$ carries a unique structure of a graded Lie algebra, defined by the so-called {\em Fr\"olicher-Nijenhuis bracket},
\[
[\cdot, \cdot]^{FN}:  \Om^k (M, TM) \times \Om^l (M, TM) \to \Om^{k+l} (M, TM)
\]
such that $\Ll$ defines an action of $\Om^\ast(M, TM)$ on $\Om^\ast(M)$, that is,
\begin{equation} \label{eq:FN-homom}
\Ll_{[K_1, K_2]^{FN}} = [\Ll_{K_1}, \Ll_{K_2}] = \Ll_{K_1} \circ \Ll_{K_2} - (-1)^{|K_1||K_2|} \Ll_{K_2} \circ \Ll_{K_1}.
\end{equation}
It is given by the following formula for $\alpha^k \in \Om^k(M)$, $\beta^l \in\Om ^l (M)$, $X_1, X_2 \in \X (M)$ \cite[Theorem 8.7 (6), p. 70]{KMS1993}:
\begin{align}
\nonumber [\alpha^k \otimes  X_1, &\beta^l \otimes  X_2]^{FN} = \alpha^k \wedge \beta^l \otimes [ X_1,  X_2]\nonumber \\
& + \alpha^k \wedge (\Ll_{X_1} \beta^l) \otimes X_2 
- (\Ll_{X_2} \alpha^k) \wedge \beta^l \otimes X_1 \label{eq:kms}\\ 
&+ (-1)^{k} \left( d \alpha^k \wedge (\imath_{X_1} \beta^l) \otimes X_2 
+ (\imath_{X_2} \alpha^k) \wedge d \beta^l \otimes X_1 \right).
\nonumber
\end{align}
In particular, for a vector field $X \in  \X(M)$ and $K \in \Om^\ast(M, TM)$ we have \cite[Theorem 8.16 (5), p. 75]{KMS1993}
\begin{align} \label{eq:liefn}
\Ll_X (K) = [X, K] ^{FN},
\end{align}
that is, the Fr\"olicher-Nijenhuis bracket with a vector field coincides with the Lie derivative of the tensor field $K \in \Om^\ast(M, TM)$. This means that $\exp(tX): \Om^\ast(M, TM) \to \Om^\ast(M, TM)[[t]]$ is the action induced by (local) diffeomorphisms of $M$. Thus, Proposition \ref{prop:homology-general} now immediately implies the following result.

\begin{theorem} \label{thm:FN-cohomology}
Let $M$ be a manifold and $K \in \Om^{2k+1}(M, TM)$ be such that $[K,K]^{FN}= 0$, and define the differential $d_K(K') := [K, K']^{FN}$. Then
\begin{enumerate}
\item $H_K^0(\Om^\ast(M, TM))$ is the Lie algebra of vector fields stabilizing $K$.
\item $H_K^{2k+1}(\Om^\ast(M, TM))$ is the space of infinitesimal deformations of $K$ within the differentials of $\Om^\ast(M, TM)$ of the form $ad_{\xi^{2k+1}}$, modulo (local) diffeomorphisms.
\end{enumerate}
\end{theorem}

\section{Fr\"olicher-Nijenhuis cohomology} \label{subs:parallel}

Suppose that $(M,g)$ is an $n$-dimensional Riemannian manifold with Levi-Civita connection $\nabla$, and $\Psi \in \Om^{2k}(M)$ is a parallel form of even degree. 
We now make the following simple but crucial observation.
The proof is given by a straightforward computation in geodesic normal coordinates.

%\commentblue{I think we need to cite   our results  that alreay appeared or in arXiv}
\begin{proposition} \label{prop:FN}
Let $\hat \Psi := \p_g \Psi \in \Om^{2k-1}(M, TM)$ with the contraction map $\p_g$ from (\ref{eq:def-partial}).
Then $\hat \Psi$ is a Maurer-Cartan element, i.e., $[\hat \Psi, \hat \Psi]^{FN} = 0$.
\end{proposition}

Thus, by the discussion in Section \ref{sec:coh}, the Lie derivative $\Ll_{\hat \Psi}$ and the adjoint map $ad_{\hat \Psi}$ are differentials on $\Om^\ast(M)$ and $\Om^\ast(M, TM)$, respectively, and for simplicity, we shall denote these by
\[
\Ll_\Psi: \Om^\ast(M) \longrightarrow \Om^\ast(M), \qquad 
ad_\Psi: \Om^\ast(M, TM) \longrightarrow \Om^\ast(M, TM),
\]
or, if we wish to specify the degree,
\begin{equation*} 
\Ll_{\Psi; l}: \Om^{l-2k+1}(M)  \longrightarrow  \Om^l(M), \qquad
ad_{\Psi; l}: \Om^{l-2k+1}(M, TM)  \longrightarrow  \Om^l(M, TM).
\end{equation*}
The cohomology algebras we denote by $H_\Psi^\ast(M)$ and $H_\Psi^\ast(TM)$ instead of $H_{\hat \Psi}^\ast(\Om^\ast(M))$ and $H_{\hat \Psi}^\ast(\Om^\ast(M, TM))$, respectively. That is, 
the $i$-th cohomologies are defined as
\begin{equation} \label{eq:def-cohom}
\begin{array}{rcl}H^i_{\Psi}(M) & := & \dfrac{\ker \Ll_{\Psi}: \Om^i(M) \to \Om^{i+2k-1}(M)}{\Im \Ll_{\Psi}: \Om^{i-2k+1}(M) \to \Om^i(M)},\\
[5mm] H^i_{\Psi}(M, TM) 
& := & 
\dfrac{\ker ad_{\Psi}: \Om^i(M, TM) \to \Om^{i+2k-1}(M, TM)}
{\Im ad_{\Psi}: \Om^{i-2k+1}(M, TM) \to \Om^i(M, TM)}.\end{array}
\end{equation}

\begin{example} \label{ex:Kaehler}
In the case of a K\"ahler manifold, using the K\"ahler form $\Psi = \om$, the differential $\Ll_\om$ on $\Om^\ast(M)$ is the complex differential $d^c := i (\bar{\p}- \p)$, 
whereas on $\Om^\ast(M, TM)$, 
$ad_\om$
 coincides with the Dolbeault differential $\bar{\p}: \Om^{p,q}(M, TM) \to \Om^{p,q+1}(M, TM)$ \cite{FN1956b}. Thus, these differentials recover well known and natural cohomology theories.
In particular, the cohomology algebras $H^\ast_\om(M)$ and $H^\ast_\om(M, TM)$ are finite dimensional if $M$ is closed.
\end{example}

Now we give some general strategies to compute $H^*_{\Psi}(M)$. 
First we summarize formulas of $\Ll_\Psi$.

\begin{lemma}[{\cite[Section 2.4]{KLS2}}]

\begin{equation*} 
\Ll_\Psi d\alpha^l = -d \Ll_\Psi \alpha^l, \quad 
\Ll_\Psi d^\ast \alpha^l = - d^\ast \Ll_\Psi \alpha^l
\quad 
\mbox{and thus} \quad 
\Ll_\Psi \triangle \alpha^l = \triangle \Ll_\Psi \alpha^l.
\end{equation*}
\begin{equation*} 
\triangle (\alpha \wedge \Psi) = (\triangle \alpha) \wedge \Psi, \qquad 
\triangle (\alpha \wedge \ast \Psi) = (\triangle \alpha) \wedge \ast \Psi,
\end{equation*}
where $\triangle$ is the Hodge Laplacian. 
As in the case of $d^\ast$, 
the formal adjoint
$\Ll_{\Psi;l}^\ast: \Om^l (M) \rightarrow \Om^{l-2k+1}(M)$ 
of 
$\Ll_{\Psi;l}: \Om^{l-2k+1}(M) \rightarrow \Om^l (M)$ 
is given by 
\begin{equation}\label{eq:formal-adjointLK}
\Ll_{\Psi;l}^\ast \alpha^l = (-1)^{n(n-l) + 1} \ast \Ll_{\Psi} \ast \alpha^l.
\end{equation}
\end{lemma}

Recall that for a closed oriented Riemannian manifold $(M,g)$ there is the {\em Hodge decomposition }of differential forms
\begin{equation} \label{Hodge-dec}
\Om^l(M) = \Hh^l(M) \oplus d\Om^{l-1} (M) \oplus d^\ast\Om^{l+1} (M),
\end{equation}
where $\Hh^l(M) \subset \Om^l(M)$ denotes the space of harmonic forms.

We define the {\em space of $\Ll_\Psi$-harmonic forms} as
\begin{equation} \label{eq:LK-harmonic}
\begin{array}{lll} \Hh_\Psi^l(M) & := & \{ \alpha \in \Om^l(M) \mid \Ll_\Psi \alpha = \Ll_\Psi^\ast \alpha = 0\}\\[3mm]
& \stackrel{(\ref{eq:formal-adjointLK})}= & \{ \alpha \in \Om^l(M) \mid \Ll_\Psi \alpha = \Ll_\Psi \ast \alpha = 0\}
\end{array}
\end{equation}
Evidently, the Hodge-$\ast$ yields an isomorphism
\begin{equation} \label{eq:ast-Hodge}
\ast: \Hh_\Psi^l(M) \longrightarrow \Hh_\Psi^{n-l}(M).
\end{equation}

Since $\Hh^l_\Psi (M) \subset \ker \Ll_{\Psi;l+2k-1}$ and $\Hh^l_\Psi (M) \cap \Im(\Ll_{\Psi;l}) = 0$, there is a canonical injection
\begin{equation} \label{eq:inject-harmonic}
\imath_l: \Hh^l_\Psi(M) \hookrightarrow H^l_\Psi(M).
\end{equation}

This is analogous to the inclusion of harmonic forms into the de Rham cohomology of a manifold, which for a closed manifold is an isomorphism due to the Hodge decomposition (\ref{Hodge-dec}). Therefore, one may hope that the maps $\imath_l$ are isomorphisms as well. It is not clear if this is always true, but we shall give conditions which assure this to be the case and show that in the applications we have in mind, this condition is satisfied.

\begin{definition} \label{def:reg-cohom}
We say that the differential $\Ll_\Psi$ is {\em $l$-regular }for $l \in \N$ if there is a direct sum decomposition
\begin{equation} \label{eq:condition-split}
\Om^l(M) = \ker (\Ll_{\Psi;l}^\ast) \oplus \Im(\Ll_{\Psi;l}).
\end{equation}
\end{definition}

A standard result from elliptic theory states that $\Ll_\Psi$ is $l$-regular if the differential operator $\Ll_{\Psi;l}: \Om^{l-2k+1}(M) \to \Om^l(M)$ is elliptic, overdetermined elliptic or underdetermined elliptic, see e.g. \cite[p.464, 32 Corollary]{Besse1987}.

The following theorem now relates the cohomology $H^\ast_\Psi(M)$ to the $\Ll_\Psi$-harmonic forms $\Hh^\ast_\Psi(M)$.

\begin{theorem}[{\cite[Theorem 2.7]{KLS2}}]

\label{thm:cohom-harmonic}
\begin{enumerate}
\item If $\Ll_\Psi$ is $l$-regular, then the map $\imath_l$ from (\ref{eq:inject-harmonic}) is an isomorphism.\item There are direct sum decompositions
\begin{eqnarray}
\label{eq:HL-decomp}
H_\Psi^l(M) & = & \Hh^l(M) \oplus H_\Psi^l(M)_d \oplus H_\Psi^l(M)_{d^\ast}\\
\label{eq:decomp-Hh_K}
\Hh_\Psi^l(M) & = & \Hh^l(M) \oplus \Hh_\Psi^l(M)_d \oplus \Hh_\Psi^l(M)_{d^\ast},
\end{eqnarray}
where $\Hh^l(M)$ is the space of harmonic $l$-forms on $M$, $H_\Psi^l(M)_d$ and $H_\Psi^l(M)_{d^\ast}$ are the cohomologies of $(d\Om^\ast(M), \Ll_\Psi)$ and $(d^\ast\Om^\ast(M), \Ll_\Psi)$, respectively, and where $\Hh_\Psi^l(M)_d := \Hh_\Psi^l(M) \cap d\Om^{l-1}(M)$, $\Hh_\Psi^l(M)_{d^\ast} := \Hh_\Psi^l(M) \cap d^\ast\Om^{l+1}(M)$. Moreover, the injective map $\imath_l$ from (\ref{eq:inject-harmonic}) preserves this decomposition, i.e.,
\[
\imath_l: \Hh_\Psi^l(M)_d \hookrightarrow H_\Psi^l(M)_d \quad \mbox{and} \quad \imath_l: \Hh_\Psi^l(M)_{d^\ast} \hookrightarrow H_\Psi^l(M)_{d^\ast}.
\]
\item
There are isomorphisms
\begin{align}
\label{eq:Hl-split1} d: H_\Psi^l(M)_{d^\ast} \to H_\Psi^{l+1}(M)_d \quad \mbox{and} \quad d^\ast: H_\Psi^l(M)_d \to H_\Psi^{l-1}(M)_{d^\ast}\\
\label{eq:Hl-split} d: \Hh_\Psi^l(M)_{d^\ast} \to \Hh_\Psi^{l+1}(M)_d \quad \mbox{and} \quad d^\ast: \Hh_\Psi^l(M)_d \to \Hh_\Psi^{l-1}(M)_{d^\ast}
\end{align}
\item If $\Ll_\Psi$ is $(l+1)$-regular and $(l-1)$-regular, then it is also $l$-regular.\\
\end{enumerate}
\end{theorem}

Next, we consider another important case. 
We call a form $\Psi \in \Om^k(M)$ {\em multi-symplectic}, if $d \Psi = 0$ and for all $v \in TM$
\begin{equation} \label{eq:Psi-nondeg}
\imath_v \Psi = 0 \Longleftrightarrow v = 0.
\end{equation}

\begin{lemma} \label{lem:elliptic multi-symplectic}
If $\Psi \in \Om^{2k}(M)$ is multi-symplectic, then the differential operator $\Ll_{\Psi;l}: \Om^{l-2k+1}(M) \to \Om^l(M)$ is overdetermined elliptic for $l = 2k-1$ and underdetermined elliptic for $l = n$.
\end{lemma}

We also can make some statement for $H_\Psi^l(M)$ for special values of $l$.

\begin{proposition}\label{prop:chomom-LK_0,n}
For a parallel form $\Psi \in \Om^{2k}(M)$, we have 
\[
H_\Psi^0(\Om^\ast(M)) \cong \Hh_\Psi^0(M) = \{ f \in C^{\infty}(M) \mid \imath_{df^\#} \Psi = 0 \},
\]

If $\Psi$ is multi-symplectic, then $\Hh_\Psi^0(M) = \Hh^0(M)$ and $H_\Psi^n (\Om^\ast(M)) \cong \Hh_\Psi^n(M) = \Hh^n(M)$.
\end{proposition}
Indeed, it can be shown that $H_\Psi^0(\Om^\ast(M))$ is infinite dimensional if $\Psi$ is not multi-symplectic.

\begin{proposition} 
\label{prop:chomom-LK 1-form}
Let $\Psi \in \Om^{2k}(M)$ be a parallel multi-symplectic form. Then
\begin{align*}
H_\Psi^{2k-1}(\Om^\ast(M)) &= \Hh^{2k-1}_\Psi(M), \\
\ker (\Ll_{\Psi; 2k}) &= \{ \alpha \in \Om^1(M) \mid \Ll_{\alpha^\#} (\ast \Psi) = 0 \quad \mbox{and} \quad d^\ast \alpha = 0\}.
\end{align*}
In particular, if $k \geq 2$ then $\ker \Ll_{\Psi; 2k} = \Hh^1_\Psi(M) \cong H^1_\Psi(M)$ and
\[
\Hh^{n-1}_\Psi(M) = \{ \alpha \in \Om^{n-1}(M) \mid \Ll_{(\ast\alpha)^\#} (\ast \Psi) = 0 \quad \mbox{and} \quad d\alpha = 0\}.
\]
\end{proposition}

The first statement is an immediate consequence of Theorem \ref{thm:cohom-harmonic} and 
Lemma \ref{lem:elliptic multi-symplectic}. 
The second and the third statements follow 
from a direct computation and an integration by parts argument.

%%%%%%%%%%%%%%%%%%%%%%%%%%%%%%%%%%%%%%%%%%%%%%%%%%%%%%%%%
\section{The Fr\"olicher-Nijenhuis cohomology of manifolds with special holonomy}\label{sec:g2}
%%%%%%%%%%%%%%%%%%%%%%%%%%%%%%%%%%%%%%%%%%%%%%%%%%%%%%%%%

On a $G_2$-manifold, there is a canonical parallel $4$-form $\ast \varphi$, 
the Hodge dual of the $G_2$-structure $\varphi$. 
We may consider the differentials $\Ll_{\ast \varphi}$ and $ad_{\ast \varphi}$. 
On closed manifolds, we obtain the following results on their cohomology groups.

\begin{theorem}\label{thm:hG2} Let $(M^7, \varphi)$ be a closed $G_2$-manifold. Then for the cohomologies $H^i_{\ast \varphi}(M^7)$ and $H^i_{\ast \varphi}(M^7, TM^7)$ defined above, the following hold.
\begin{enumerate}
\item There is a Hodge decomposition
\begin{align*}
H^i_{\ast \varphi}(M^7) =\;& \Hh^i(M^7) \oplus (H^i_{\ast \varphi}(M^7) \cap d\Om^{i-1}(M^7)) \\ &\oplus (H^i_{\ast \varphi}(M^7) \cap d^\ast\Om^{i+1}(M^7)),
\end{align*}
where $\Hh^i(M^7)$ denotes the spaces of harmonic forms.
\item The Hodge-$\ast$ induces an isomorphism $\ast: H^i_{\ast \varphi}(M^7) \to H^{7-i}_{\ast \varphi}(M^7)$.
\item $H^i_{\ast \varphi}(M^7) = \Hh^i(M^7)$ for $i = 0,1,6,7$. For $i = 2,3,4,5$, $H^i_{\ast \varphi}(M^7)$ is infinite dimensional.
\item $\dim H^0_{\ast \varphi}(M^7, TM^7) = b^1(M^7)$; in particular, $H^0_{\ast \varphi}(M^7) = 0$ if $M^7$ has full holonomy $G_2$.
\item $\dim H^3_{\ast \varphi}(M^7, TM^7) \geq b^3(M^7) > 0$, as it contains all torsion free deformations of the $G_2$-structure modulo deformations by diffeomorphisms.
\end{enumerate}
\end{theorem}

In \cite[Theorem 3.5]{KLS2}, 
we give a precise description of the cohomology ring $H^\ast_{\ast \varphi}(M^7)$. 

\begin{remark}
On a $\Sp$-manifold, there is also a canonical parallel $4$-form 
and we obtain the similar results. For more details, see \cite[Theorem 4.2]{KLS2}. 

Recently, we also computed $H^i_{\Phi}(M)$
for the real part of a holomorphic volume form $\Psi$ in $4n$-dimensional Calabi-Yau manifolds 
in \cite{KLS3}. 
When $n=1$, it is isomorphic to the de Rham cohomology. 
When $n \geq 2$, as in the $G_2$ and $\Sp$-case, 
it is regular again, and all summands involved other than the harmonic forms are infinite dimensional.
\end{remark}

\begin{proof}[Outline of the proof of Theorem \ref{thm:hG2}]
We begin by showing the $l$-regularity of $\Ll_{\ast \varphi}$. For $l < 3$ and $l > 7$, this is obvious as then $\Ll_{\ast \varphi; l} = 0$. By Lemma \ref{lem:elliptic multi-symplectic}, $\Ll_{\ast \varphi,l}$ is overdetermined elliptic for $l = 3$ and underdetermined elliptic for $l=7$, whence $\Ll_{\ast \varphi}$ is also $3$- and $7$-regular. 

By a simple calculation, it follows that 
$\Ll_{\ast \varphi,l}$ is overdetermined elliptic for $l = 4$ and underdetermined elliptic for $l=6$, whence $\Ll_{\ast \varphi}$ is $4$-regular and $6$-regular. 
Thus it is also $5$-regular by Theorem \ref{thm:cohom-harmonic}(4). Therefore, the $l$-regularity of $\Ll_{\ast \varphi,l}$ for all $l$ is established, whence by Theorem \ref{thm:cohom-harmonic}(1), $H^l_{\ast \varphi}(M) = \Hh^l_{\ast \varphi}(M)$.

For $l = 0,7$, $\Hh^l_{\ast \varphi}(M) \cong \Hh^l(M)$ by Proposition \ref{prop:chomom-LK_0,n}.

For $l = 1$, $H^1_{\ast \varphi}(M) = \ker \Ll_{\ast \varphi}|_{\Om^1(M)}$. Thus, by Proposition \ref{prop:chomom-LK 1-form}, $\alpha \in H^1_{\ast \varphi}(M)$ implies that $\Ll_{\alpha^\#}(\varphi) = 0$, which in turn implies that $\alpha^\#$ is a Killing vector field. Since a $G_2$-manifold is Ricci flat, it follows by Bochner's theorem that $\alpha^\#$ is parallel, whence so is $\alpha$. In particular, $\alpha$ is harmonic, showing that $H^1_{\ast \varphi}(M) = \Hh^1(M)$. For $l = 6$, we have $H^6_{\ast \varphi}(M) = \ast H^1_{\ast \varphi}(M) = \ast \Hh^1(M) = \Hh^6(M)$. This shows that $\Hh^l_{\ast \varphi}(M) \cong \Hh^l(M)$ for $l = 1, 6$.

Next, for $l = 2$, we have $\Hh^2_{\ast \varphi}(M)_d = 0$ by (\ref{eq:Hl-split1}). Thus, we need to determine
\[
\Hh^2_{\ast \varphi}(M)_{d^\ast} = \{ \alpha^2 \in d^\ast \Om^3(M) \mid d^\ast (\alpha^2 \wedge \ast \varphi) = 0\}.
\]
We can investigate this space in detail 
by the irreducible decomposition of $\Om^*(M)$ under the $G_2$-action 
and the Hodge decomposition. 
Then we can prove that $\Hh^2_{\ast \varphi}(M)_{d^\ast}$ is isomorphic to 
an infinite dimensional function space. 
We can prove the case of $l=3$ similarly. 

Again, since $\ast: \Hh^l_{\ast \varphi}(M) \to \Hh^{7-l}_{\ast \varphi}(M)$ is an isomorphism, the assertions for $l = 4, 5$ follow.\\

Next, we consider $H^\ast_{\ast \varphi}(M^7, TM^7)$. 
First, note the following.

\begin{lemma}\label{lem:uniq} 
Let $V$ be an oriented $7$-dimensional vector space, and let $\Lambda^3 _{G_2} V^\ast \subset \Lambda^3 V^\ast$ be the set of $G_2$-structures on $V$. 
By definition, the group $GL^+(V)$ of orientation preserving automorphisms of $V$ act 
transitively on $\Lambda^3 _{G_2} V^\ast$ so that
$
\Lambda^3 _{G_2} V^\ast = GL^+(V)/G_2
$.
Then the map
\begin{eqnarray*} & \frak{C}: \Lambda^3 _{G_2} V^\ast \longrightarrow \Lambda ^3 V^\ast \otimes V, & \qquad \varphi \longmapsto \p_{g_\varphi} (\ast_{g_\varphi} \varphi)
\end{eqnarray*}
is a $GL^+(V)$-equivariant injective immersion.  
Here, $\p_g$ is the map from (\ref{eq:def-partial}), and 
$g_\varphi$ denotes the metric induced by $\varphi$.
\end{lemma}

\begin{proposition} \label{prop:H0MTM}
\begin{equation*}
H^0_{\ast \varphi}(M^7, TM^7) = \{ X \in {\frak X}(M^7) \mid \Ll_X \varphi = 0\} 
=\{ X \in {\frak X}(M^7) \mid \nabla X = 0\}.
\end{equation*}
\end{proposition}

This proposition implies the 4th part of Theorem \ref{thm:hG2}. 

\begin{proof}
Let $X \in {\frak X}(M^7)$ be a vector field, $p \in M^7$ and denote by $F_X^t$ the local flow along $X$, defined in a neighborhood of $p$. Then because of the pointwise equivariance of $\frak{C}$ we have
\[
(F_X^t)^\ast \Big(\p_ g \ast \varphi \Big)_{F_X^t(p)} = (F_X^t)^\ast \Big(\frak{C}(\varphi)_{F_X^t(p)}\Big) = \frak{C}\Big((F_X^t)^\ast(\varphi_{F_X^t(p))}\Big)
\]
and taking the derivative at $t=0$ yields
\begin{equation} \label{eq:Lie-equiv}
\Ll_X (\p_g \ast \varphi)_p = \Ll_X (\frak{C}(\varphi))_p = \frak{C}_* (\Ll_X \varphi)_p.
\end{equation}
Now $\Ll_X (\p_g \ast \varphi) = [X, \p_g \ast \varphi]^{FN}$, and since $\frak C$ is an immersion by Lemma \ref{lem:uniq}, it follows that $X \in H^0_{\ast \varphi}(M^7, TM^7) = \ker ad_{\p_g \ast \varphi}$ iff $\Ll_X\varphi = 0$. 

Since $\varphi$ uniquely determines the Riemannian metric $g_\varphi$ on $M^7$, any vector field satisfying $\Ll_X \varphi = 0$ must be a Killing vector field. 
Since $M^7$ is closed, the Ricci flatness of $G_2$-manifolds and Bochner's theorem imply that $X$ must be parallel, showing that in this case, $\dim H^0_\Phi(M^7, TM^7) = b^1(M^7)$.
\end{proof}

It was shown in \cite[Theorem 1.1]{KLS} that 
a $G_2$-structure $\varphi'$ is torsion-free if and only if 
$[\chi_{\varphi'}, \chi_{\varphi'}]^{FN}=0$, where 
$\chi_{\varphi'} := \frak{C}(\varphi') = \p_{g_{\varphi'}} \ast_{g_{\varphi'}} \varphi' \in \Om^3(M^7, TM^7)$.
Therefore, 
for a family of torsion-free $G_2$-structures $\{ \varphi_t \}$ with $\varphi_0 = \varphi$, 
we have 
\begin{align*} 
0 = 
\left. \dfrac d{dt} \right|_{t=0} [\chi_{\varphi_t}, \chi_{\varphi_t}] = 2 \left[ \chi_{\varphi_0}, \left. \dfrac d{dt} \right|_{t=0}\chi_{\varphi_t} \right]
&= 2 \left[ \chi_{\varphi_0}, \frak{C}_* (\dot \varphi_0)\right],
\end{align*}
so that $\dot \varphi_0  \in \Om^3(M^7)$ is a torsion free infinitesimal deformation of $\varphi_0$ 
iff $\frak{C}_* (\dot \varphi_0) \in \ker(ad_{\chi_{\varphi_0}}: \Om^3(M^7, TM^7) \to \Om^6(M^7, TM^7))$. Since ${\frak C}$ is an immersion and hence $\frak{C}_*$ injective by Lemma \ref{lem:uniq}, we have an isomorphism
\begin{align*}
\{ \text{torsion free\;}&\text{infinitesimal deformations of $\varphi_0$}\} \stackrel{\frak{C}_*} \cong\\
\nonumber
&\ker\Big(ad_{\chi_{\varphi_0}}: \Om^3(M^7, TM^7) \to \Om^6(M^7, TM^7)\Big) \cap \Im(\frak{C}_*).
\end{align*}

Observe that by (\ref{eq:Lie-equiv})
\begin{equation*}
\frak{C}_* (\Ll_X \varphi_0) = \Ll_X({\frak C}(\varphi_0)) = [X, \chi_{\varphi_0}]^{FN} = - ad_{\chi_{\varphi_0}}(X),
\end{equation*}
whence there is an induced inclusion
\begin{align*}
\dfrac{\{ \text{torsion free infinitesimal deformations of $\varphi_0$}\}}{\{\text{trivial deformations of $\varphi_0$}\}}\xhookrightarrow{\quad \mbox{$\frak{C}_*$}\quad} H^3_{\varphi_0}(M^7, TM^7).
\end{align*}
This implies the 5th part of Theorem \ref{thm:hG2}. 
\end{proof}

\section{Strongly homotopy  Lie algebra  associated  with
associative submanifolds}\label{sec:linfty}

In this  section  we  assign to each associative  submanifold in a $G_2$-manifold 
an $L_\infty$-algebra, using the  Fr\"olicher-Nijenhuis bracket  and   Voronov's derived bracket construction of $L_\infty$-algebras.
The main  purpose of this section  is to explain  the  motivation that led us to study 
the Fr\"olicher-Nijenhuis bracket on $G_2$- manifolds. 
We refer the reader to \cite{FLSV2018} for detailed and general treatment of the theory discussed  here.  

\subsection{Voronov's construction of  $L_\infty$-algebras}\label{subs:constr}

Strongly homotopy Lie algebras, also called  $L_\infty$-algebras, were defined by Lada and Stasheff in \cite{LS1993}, see  also \cite{Voronov2005} for a  historical account.
In \cite{Voronov2005}  Voronov suggested   a  relatively simple   method to construct  an $L_\infty$-algebra based  on algebraic data, now called V-data. 
A {\it $V$-data}   is  a  quadruple $(L,  P, \a, \triangle)$, where 
\begin{enumerate}
\item $L$ is a $\Z_2$-graded Lie algebra   $L =  L_0 \oplus L_1$ (we denote its bracket by $[.,.]$),
\item $\a$-  an abelian Lie subalgebra  of $L$,
\item $P$ : $L \to \a$  is a projection whose kernel is a Lie subalgebra of $L$,
\item  $\triangle  \in (\ker P) \cap L_1$ is an element such that $[\triangle, \triangle] = 0$.
\end{enumerate}

When $\triangle$ is an arbitrary element of $L_1$ instead of $\ker (P)\cap L_1$, we refer to $(L, \a, P,\triangle)$ as a
{\it curved V-data}.

Recall that a
$(k, l)$-shuffle is a permutation of indices $1, 2, \cdots , k+l$ such that $\sigma(1) < \cdots <\sigma(k)$ and $\sigma(k+1) < \cdots  <\sigma(k+l)$.

\begin{definition}[{\cite[Definition 1]{Voronov2005}}] \label{def:linfty}
A vector space $V = V_0 \oplus V_1$ endowed with a sequence of
odd $n$-linear operations  $\m_n$, $n = 0, 1, 2, 3, \cdots$,  is
a {\it strongly homotopy Lie algebra or $L_\infty$-algebra} if: (a) all operations are
symmetric in the $\Z_2$-graded sense:
$$\m_n (a_1, \cdots , a_i, a_{i+1}, \cdots , a_n) = (-1)^{\bar a_i\bar a_{i+1}}\m_n ( a_1, \cdots , a_{i+1}, a_i, \cdots, a_n), $$
and (b) the ``generalized Jacobi identities''
$$\sum_{k+l=n}\sum_{(k, l)-\rm{shuffles}} (-1)^{\alpha}\m_{l+1} (\m _k (a_{\sigma(1)}, \cdots , a_{\sigma(k)}), a_{\sigma(k+1)}, \cdots, a_{\sigma(k+l)}) = 0 $$
hold for all $n = 0, 1, 2, \cdots$. Here  $\bar a$ is the  degree  of $a \in V$  and $(-1)^{\alpha}$ is the sign prescribed by the sign
rule for a permutation of homogeneous elements $a_1, \cdots , a_n \in V$.
\end{definition}

Henceforth symmetric will mean $\Z_2$-graded symmetric.

A 0-ary bracket is just a distinguished element  $\Phi$ in $V$. 
We  call the $L_\infty$-algebras with $\Phi = 0$ {\it strict}. In this  case $\m_1 ^2 = 0$ and  we also write  $d$ instead of $\m_1$.
For strict $L_\infty$-algebras, the  first  three ``generalized Jacobi identities'' simplify to
$$d^2 a = 0,$$
$$d\m_2(a, b) + \m_2(da, b) + (-1)^{\bar a\bar b} \m_2(db, a) = 0,$$
$$d\m_3(a, b, c) + \m_2(\m_2(a, b), c) + (-1)^{\bar b \bar c}\m_2(\m_2(a, c), b)
 + (-1)^{\bar a(\bar b+\bar c)}\m_2(\m_2(b, c), a)$$
 $$ +\m_3(da, b, c) + (-1)^{\bar a\bar b} \m_3(db, a, c) + (-1)^{(\bar a+\bar b)\bar c} \m_3(dc, a, b) = 0.$$

\begin{proposition}\label{prop:voronov}(\cite[Theorem 1, Corollary 1]{Voronov2005}). Let $(L, \a, P, \triangle)$ be a curved V-data. Then $\a$ is a curved  $L_\infty$-algebra for the multibrackets 
$$ \m_n(a_1, \cdots , a_n) = P[\cdots [[\triangle, a_1], a_2],\cdots , a_n].$$
We obtain a  strict $L_\infty$-algebra exactly when $\triangle \in \ker (P)$.
\end{proposition}

\begin{remark}\label{rem:curv}  Usually in the literature  a strict $L_\infty$-algebra  is called {\it an $L_\infty$-algebra}, which we    also adopt  in this paper.
\end{remark}

\subsection{$L_\infty$-algebra associated  with an associative submanifold $L \subset M^7$}\label{subs:linfty}
Let $L$ be a   closed submanifold in   a  Riemannian $M$.
There exists  an open  neighborhood $N_\eps L$ of  the zero section $L$ in  the normal bundle $NL$ such that  the  exponential mapping $Exp: N_\eps L \to  M$ is a 
local diffeomorphism.

Given  such  an open neighborhood  $N_\eps L$,  we consider the
pullback operator $Exp ^*: \Om^* (M, TM) \to \Om^*(N_\eps L, TNL)$
$$ Exp ^* ( K)_x(X_1, \cdots ,X_k) = (Exp_*)_x^{-1}(K_{Exp(x)}((Exp_*)_x (X_1), \cdots , (Exp_*)_x (X_k))$$
for  $x \in  N_\eps L$, $X_i  \in T_x  N_\eps L$ and  $K \in \Om^k (M, TM)$ 
given by  \cite[8.16, p. 74]{KMS1993}.
It is known that \cite[Theorem 8.16 (2), p. 74]{KMS1993}
\begin{equation}
[Exp ^* (K),  Exp ^* (L)] ^{FN} = Exp ^* ([K, L] ^{FN}).\label{eq:exp}
\end{equation}

% For  any  $x \in  L$  and  any $X \in  T_xL$  we define
%the unique lift $\hat X \in  \X (N_x L)$ such that $\hat X (0 ) = X$,  using the canonical  flat connection on the fiber  $N_x L$.
Let $\pi$ denote  the  projection $NL \to L$. 
For each  section $X \in \Gamma (NL)$, 
define the vector field $\hat X$ on $N_\eps L \subset NL$ by 
the restriction of the vertical lift of $X$ to $N_\eps L$. That is, 
$$
\hat X (y) = \left. \frac{d}{d t} \left(y + t X (\pi (y)) \right) \right|_{t=0} 
\qquad 
\mbox{ for } y \in N_\eps L. 
$$

Let $\Om^*(N_\eps L, TNL)$ be the space  of all smooth   $TNL$-valued    forms on  $N_\eps L$.
We define a linear embedding
$I:  \Om^*(L, NL) \to \Om^*(N_\eps L, TNL)$  by 
$$ I( \phi \otimes  X ): =  \pi^*(\phi) \otimes  \hat X$$
and extend  it linearly on the whole $\Om^*(L, NL)$.

  Let $P$ denote the   composition  of   the     restriction
$r : \Om^*(N_\eps L, TNL) \to  \Om^*(L, TNL)$  and the projection 
$Pr^N: \Om^* (L, TNL) \to  \Om^* (L, NL)$
defined via the  decomposition $TNL_{|L} = NL \oplus TL$.  
Set 
$$\tilde P : =  I\circ P: \Om^*(N_\eps L, TNL) \to \Om ^* (N_\eps L, TNL).$$

\begin{lemma}\label{lem:ext}  The image of  the map $\tilde P$ is an abelian  subalgebra of the $\Z_2$-graded Lie algebra $(\Om^* (N_\eps L, TNL), [., .] ^{FN})$. The space $\ker  \tilde P$  is closed  under  the Fr\"olicher-Nijenhuis  bracket.
\end{lemma}

\begin{proof}    Note that the image of $\tilde P$ is   equal to the image of $I$. To   prove the first  assertion  of Lemma \ref{lem:ext}  it  suffices to   prove  that
\begin{equation}
[\pi^*(\alpha_1) \otimes \hat X_1, \pi^*(\alpha_2)\otimes \hat X_2]^{FN} = 0\label{eq:van1}
\end{equation}
for  any $X, Y \in \Gamma(NL)$  and   any $\alpha_1, \alpha_2 \in 
\Om^* (L)$.  

Using (\ref{eq:kms}) and taking into account  the following  identities
$$[\hat X_1 , \hat X_2]^{FN}   = 0, $$
$$i_{\hat X_1} \pi ^* (\alpha_2) = 0 = i_{\hat X_2} \pi^* (\alpha_1), $$
$$ d \pi ^* (\alpha_1) = \pi^* (d \alpha_1),$$
for any $\alpha_1, \alpha_2 \in \Om ^*(L)$  and any  $X_1, X_2 \in \X (L)$
we  obtain (\ref{eq:van1})  immediately.

Now let us prove the  second assertion  of Lemma \ref{lem:ext}.  Since $I$ is injective, we have  $\ker \tilde P = \ker P$. Moreover,
$\ker P$ is generated by  the      $TNL$-valued   differential forms
$\alpha \otimes X$  such that $X(x) \in T_xL$ for all $x\in L$.  
Assume that  $\alpha_1 \otimes X_1, \alpha_2 \otimes  X_2 \in \ker P$.
Using (\ref{eq:kms})  and the  fact that  if $X_1, X_2 \in \X (N_\eps L)$ and $(X_1) _{| L} \in \X(L),  (X_2) _{|L} \in \X(L)$  then
$$[X_1, X_2]_{|L}  \in  \X(L),$$
we obtain immediately 
$$[\alpha_1\otimes X_1, \alpha_2 \otimes X_2] \in \ker  P = \ker \tilde P.$$
This completes  the proof of Lemma \ref{lem:ext}. 
\end{proof}

\begin{theorem}\label{thm:linfty} Assume that  $L$ is an associative 
submanifold of a $G_2$-manifold $(M^7, \varphi)$. 
There is an $L_\infty$-algebra structure  on  the space
$\Om^*(L, NL)$ given by the following family of  graded multi-linear maps
$$\m _k: \Om^*(L, NL) ^{ \otimes  k} \to \Om^*(L, NL)$$
$$\m_k (\om_1, \cdots,  \om_k) = P [\cdots [ [ Exp ^*(\chi), I (\om_1)]^{FN}, I(\om_2)]^{FN}\cdots , I(\om_k)]^{FN},$$
where $\chi = \p_g * \varphi \in \Om^3(M,TM)$  
with $\p_g$ from (\ref{eq:def-partial}). 
\end{theorem}

\begin{proof}  By  Proposition \ref{prop:FN} and using (\ref{eq:exp}),  we have
\begin{equation}\label{eq:i4} 
[Exp ^*(\chi), Exp ^* (\chi)]^{FN} = 0.
\end{equation}

\begin{lemma}\label{lem:ass}
A submanifold  $L$  in  a $G_2$-manifold $(M^7, \varphi)$ is associative iff
$$ Exp ^*(\chi) \in \ker \tilde P.$$
\end{lemma}

\begin{proof} It is known that a   3-submanifold  $L$ in a $G_2$-manifold
is  associative iff  $\chi_{| L}  = 0$ \cite{HL1982}. Since  $\chi(x\wedge y\wedge z) $ is orthogonal to each $x, y, z$, it follows  that $L$ is associative, iff
$Pr^N (\chi_{| L}) = 0\in \Om ^*(L, NL)$, where 
$Pr^N$ 
is the orthogonal projection from $TM |_L$ to the normal bundle of $L$.
This implies  Lemma \ref{lem:ass}, taking into account the injectivity of
$I$  and  $Exp^*$.
\end{proof}

Lemmas \ref{lem:ext}, \ref{lem:ass} and the equation (\ref{eq:i4})  imply  that
$(\Om^*(N_\eps L, TNL), I(\Om^*(L, NL)),\\ \tilde P,
 Exp^*( \chi))$ is a $V$-data.   This  and Proposition \ref{prop:voronov}   yield Theorem  \ref{thm:linfty}  immediately. 
\end{proof}

\begin{lemma}\label{lem:mk}  Let  $V_1, \cdots  , V_k \in \Gamma (NL) = \Om^0(L, NL)$. Then
$$\m_k (V_1, \cdots, V_k) =  P (\Ll_{I(-V_1)} \cdots \Ll_{I(-V_k)} (Exp ^* (\chi))).$$
\end{lemma}
\begin{proof}
Let $V_i \in \Gamma (NL)$. Then $\{ I(V_i)\}$    are mutually commuting vector fields on $N_\eps L$.  Using (\ref{eq:liefn})  and
noting that 
$$[Exp ^*(\chi), I(V_i)]  ^{FN} =  [I(-V_i), Exp ^*(\chi)]^{FN}$$
 we obtain Lemma \ref{lem:mk} immediately.
\end{proof}

We shall denote  the  map $\m_1:\Om^*(L, NL) \to \Om^{*+3}(L, NL) $ by $d_L$.   
Since $\dim L = 3$, $d_L$ is non-trivial only  on $\Om ^0 (L, NL)$.

\begin{remark}\label{rem:deform} Using  the formal  deformation theories developed in \cite{LO2016} and \cite{LS2014} it is not hard  to see  that
the  $L_\infty$-algebra  associated to a closed associative  submanifold $L$  encodes  the formal and smooth   associative deformations of  $L$. This and further generalization have been 
considered  in \cite{FLSV2018}. 
A search  for an $L_\infty$-algebra  associated   to each  
associative  submanifold in $G_2$-manifolds 
led us to discover the role  of the Fr\"olicher-Nijenhuis bracket 
in  geometry of $G_2$-manifolds. 
\end{remark}

%%%%%%%%%%%%%%%%%%%%%%%%%%%%%%%%%%%%%%%%%%%%%%%%%%%%%%%%%

\end{document}